\documentclass[12pt]{amsart}
\usepackage{amssymb, eucal, amsfonts, amsmath, xy,latexsym, bbm}
\usepackage{mathrsfs}

\usepackage[margin=2cm]{geometry}
\usepackage[pdftex,colorlinks,hypertexnames=false,linkcolor=black,urlcolor=black,citecolor=black]{hyperref}

\numberwithin{equation}{section}

\newtheorem{Theorem}{Theorem}[section]
\newtheorem*{Theorem*}{Theorem}
\newtheorem*{Corollary*}{Corollary}

\newtheorem{Lemma}[Theorem]{Lemma}
\newtheorem{Proposition}[Theorem]{Proposition}
\newtheorem{Corollary}[Theorem]{Corollary}

\theoremstyle{definition}

\theoremstyle{remark}
\newtheorem{Remark}[Theorem]{Remark}
\newtheorem*{Remark*}{Remark}

\newcommand{\Z}{\mathbb{Z}}

\renewcommand{\k}{\mathbbm{k}}
\renewcommand{\Z}{\mathbb{Z}}

\newcommand{\g}{\mathfrak{g}}

\newcommand{\gl}{\mathfrak{gl}}
\newcommand{\h}{\mathfrak{h}}
\renewcommand{\c}{\mathfrak{c}}
\renewcommand{\t}{\mathfrak{t}}

\renewcommand{\sl}{\mathfrak{sl}}

\newcommand{\psl}{\mathfrak{psl}}
\newcommand{\z}{\mathfrak{z}}

\newcommand{\n}{\mathfrak{n}}

\newcommand{\D}{\mathcal{D}}
\renewcommand{\O}{\mathcal{O}}

\newcommand{\LL}{\mathcal{L}}
\newcommand{\N}{{\mathcal{N}}}

\newcommand{\ad}{\operatorname{ad}}
\newcommand{\Ad}{\operatorname{Ad}}

\newcommand{\Lie}{\operatorname{Lie}}

\newcommand{\SL}{\operatorname{SL}}
\newcommand{\GL}{\operatorname{GL}}

\newcommand{\Sp}{\operatorname{Sp}}

\newcommand{\Der}{\operatorname{Der\,}}

\newcommand{\la}{\langle}
\newcommand{\ra}{\rangle}

\begin{document}

\newenvironment{changemargin}[1]{%
  \begin{list}{}{%
    \setlength{\topsep}{0pt}%
    \setlength{\topmargin}{#1}%
    \setlength{\listparindent}{\parindent}%
    \setlength{\itemindent}{\parindent}%
    \setlength{\parsep}{\parskip}%
  }%
  \item[]}{\end{list}}

\parindent=0pt
\addtolength{\parskip}{0.5\baselineskip}

\title{Sandwich elements
and the Richardson property}

\author{Alexander Premet}
\address{Department of Mathematics, The University of Manchester, Oxford Road, M13 9PL, UK}
\email{alexander.premet@manchester.ac.uk}
\pagestyle{plain}
\begin{abstract}
A restricted Lie algebra $\LL$ over an algebraically closed field $\k$ of characteristic $p>0$ is said to possess the Richardson property if there exists a finite dimensional faithful restricted $\LL$-module $V$ with associated representation $\rho\colon \LL\to \gl(V)$ such that $\gl(V)=\rho(\LL)\oplus R$ where $R$ is a subspace of $\gl(V)$ such that  $[\rho(\LL),R]\subseteq R$. In \cite{P99}, the author conjectured that if $\LL$ has the Richardson property with respect to an irreducible restricted  $\LL$-module $V$ then there exists a reductive algebraic group $G$ over $\k$ such that $\LL\cong\Lie(G)$
as restricted Lie algebras.
In this note we confirm this conjecture under the assumption that $p>3$. This assumption is needed since our proof relies in a crucial way
on the classification of Lie algebras without strong degeneration obtained in \cite{P87a} for $p>5$ and in \cite{P86} for $p=5$.
\end{abstract}
\maketitle
\begin{center}
{\it In memory of Georgia Benkart}
\end{center}


\section{Preliminaries and the statement of the main result}\label{intro}
\subsection{Restricted Lie algebras}\label{restricted}
Let $\LL$ be a finite dimensional restricted Lie algebra over an algebraically closed field $\k$ of characteristic $p>0$. We denote by $[p]\colon\,\LL\to\LL$, $\,x\mapsto x^{[p]},$ the $p$-operation on $\LL$. An element $x\in\LL$ is said to be {\it nilpotent} if $x^{[p]^N}=0$ for some $N\gg 0$. It follows from Jacobson's formula (expressing $(x+y)^{[p]}-x^{[p]}-y^{[p]}$ as a sum of Lie monomials of length $p$ in $x,y\in\LL$) that 
for any $d\in\Z_{\ge 0}$ the map $\pi_d\colon\,\LL\to\LL, \ x\mapsto x^{[p]^d}$ is a morphism of affine varieties given by a collection of homogeneous polynomial functions of degree $p^d$ on $\LL$. From this it follows that
the set $\N(\LL)$ of all nilpotent elements of $\LL$ is a Zariski closed conical subset of $\LL$.
We say that $x\in\LL$ is {\it semisimple}
if it belongs to the $\k$-span of all  $x^{[p]^i}$ with $i\ge 1$. 
If $N\in\Z_{\ge 0}$ is sufficiently large, then the set $\LL_{ss}$ of all semisimple elements of $\LL$ coincides with the image of the morphism $\pi_N$. Given $x\in\LL$ we write $\la x\ra_p$ for the $\k$-span of all $x^{[p]^i}$ with $i\ge 0$. It is well--known that there exist unique $x_s,x_n\in\la x\ra_p$ such that $x_s\in\LL_{ss}$, $x_n\in\N(\LL)$, and $x=x_s+x_n$.

We say that an element $x\in\LL$ is {\it toral} if $x^{[p]}=x$. A Lie subalgebra $\t$ of $\LL$ is called {\it toral} (or a {\it torus})
if it consists of semisimple elements of $\LL$. 
Toral subalgebras of $\LL$ are abelian and their maximal dimension, denoted ${\rm MT}(\LL)$, is an important invariant of $\LL$. Furthermore, it is well--known (and straighforward to see) that any toral subalgebra $\t$ of $\LL$ has a $\k$-basis consisting of toral elemnts of $\LL$. In other words, the set $\t^{\rm tor}\,:=\,\{t\in \t\,|\,\,t^{[p]}=t\}$ is an $\mathbb{F}_p$-form of $\t$.

We denote by $e=e(\LL)$ the smallest nonnegative integer for which there is a nonempty Zariski open subset $U=U(e)$ of $\LL$ such that $\pi_e(U)\subseteq \LL_{ss}$.  It is proved in \cite{P87b}  that if $\t\subseteq \LL$ is a torus of dimension ${\rm MT}(\LL)$ then the centraliser $\c_\LL(\t)$ is a Cartan subalgebra of minimal dimension in $\LL$ and the equality 
$$\dim\LL-\dim\overline{U(e)}\,=\,\dim\c_{\LL}(\t)-\dim\t$$ holds. From this it is immediate that
$e(\LL)=0$ if and only if $\LL$ contains a toral Cartan subalgebra.

According to \cite{P90}, if ${\rm MT}(\LL)=s$ then there exist
homogeneous polynomial functions $\psi_0,\ldots,\psi_{s-1}\in\k[\LL]$ such that $\deg \psi_i=p^{s+e}-p^{i+e}$ for $0\le i\le s-1$ and
$$x^{[p^{s+e}]}\,=\,\,\sum_{i=0}^{s-1}\,\psi_i(x)\cdot x^{[p]^{i+e}}\ \qquad\ (\forall\,x\in \LL).$$
 By \cite{P90, P03b}, each $\psi_i$ has the property that $\psi_i(x^{[p]})=\psi_i(x)^p$ for all $x\in\LL$ and is invariant under the natural action of the restricted automorphism group
${\rm Aut}(\LL,[p])$ on $\k[\LL]$. This implies that the variety $\N(\LL)$ coincides with the zero locus of $\psi_0,\ldots, \psi_{s-1}$.
Since $\t\cap\N(\LL)=\{0\}$ for any $s$-dimensional torus $\t$ of $\LL$, the Affine Dimension Theorem yields
that all irreducible components of the nilpotent cone $\N(\LL)$  have dimension equal to
$\dim\LL-{\rm MT}(\LL)$. As a consequence, the homogeneous polynomial functions $\psi_0,\ldots,\psi_{s-1}$ form a regular sequence in $\k[\LL]$ and, in particular,  are algebraically independent. In \cite{P90}, the author conjectured that the variety $\N(\LL)$ is irreducible for any finite dimensional restricted Lie algebra $\LL$. This conjecture is verified in many cases, but is still open, in general.

Let $\t$ be a maximal toral subalgebra of $\LL$. The adjoint action of $\t$ on $\LL$ gives rise to a Cartan decomposition
$$\LL\,=\,\c_{\LL}(\t)\oplus\,\sum_{\gamma\in \Gamma}\,\LL_\gamma,\qquad\quad\LL_\gamma\,=\,\{x\in\LL\,|\,\,[t,x]=\gamma(t)\cdot x\ \, \mbox{for all}\ \, t\in\t\}.$$ Here  
$\Gamma=\Gamma(\LL,\t)$ is the set roots of $\LL$ with respect to $\t$, a finite subset of 
$\t^*$ consisting of {\it nonzero} ${\mathbb F}_p$-valued linear functions on $\t^{\rm tor}$, and $\LL_\gamma\ne\{0\}$ is the root subspace of $\LL$ corresponding to $\gamma\in\Gamma$. 
Since $\t$ is a maximal torus of $\LL$ its centraliser $\c_{\LL}(\t)$ is
a Cartan subalgebra of $\LL$. Given $\mathbb{F}_p$-linearly independent roots $\gamma_1,\ldots,\gamma_d\in\Gamma\subset (\t^{\rm tor})^*$ we denote by $\Gamma(\gamma_1,\ldots,\gamma_d)$ the set of all $\gamma=a_1\gamma_1+\cdots a_d\gamma_d\in\Gamma$ with $a_i\in{\mathbb F}_p$, and set $$\LL(\gamma_1,\ldots,\gamma_d)\,:=\,\c_{\LL}(\t)\oplus\textstyle{\sum}_{\gamma\in
\Gamma(\gamma_1,\ldots,\gamma_d)}\,\LL_\gamma.$$ 
The subspace $\t(\gamma_1,\ldots,\gamma_d)\,:=\,\bigcap_{i=1}^d\ker\gamma_i$ is a subtorus of codimension $d$ in $\t$. Using the perfect pairing between 
$\t^{\rm tor}$ and $(\t^{\rm tor})^*$ it is easy to check that
$\LL(\gamma_1,\ldots,\gamma_d)$ coincides with the centraliser of 
$\t(\gamma_1,\ldots,\gamma_d)$ in $\LL$. In particular, $\LL(\gamma_1\,\ldots,\gamma_d)$ is a restricted Lie subalgebra of $\LL$. Conversely, if $\t_0$ is a subtorus of $\t$, then $\c_{\LL}(\t_0)=\LL(\gamma_1,\ldots,\gamma_d)$ for some $\mathbb{F}_p$-linearly independent roots $\gamma_1,\ldots,\gamma_d\in \Gamma$ such that
$\t_0\subseteq \t(\gamma_1,\ldots,\gamma_d)$.  The restricted Lie algebra $$\LL[\gamma_1,\ldots,\gamma_d]\,:=\,
\LL(\gamma_1,\ldots,\gamma_d)/{\rm rad}\big(\LL(\gamma_1,\ldots,\gamma_d)\big)$$ is referred to as a $d$-{\it section} of $\LL$. It is straightforward to see that ${\rm MT}\big(\LL[\gamma_1,\ldots,\gamma_d]\big) \le d$.

\subsection{Motivation for the Richardson property}\label{supp}
Given $\chi\in \LL^*$ we denote by $I_\chi$ the two-sided ideal of the universal enveloping algebra $U(\LL)$ generated by all elements
$x^p-x^{[p]}-\chi(x)^p\cdot 1$ with $x\in \LL$ (these elements are central in $U(\LL)$). The 
factor-algebra 
$U_\chi(\LL):=U(\LL)/I_\chi$ is called the {\it reduced enveloping algebra} associated with $\chi$. It is well--known that for any irreducible $\LL$-module $V$ there exists a unique linear function $\chi=\chi_V\in\LL^*$ such that $(x^p-x^{[p]})_V=\chi(x)^p\cdot {\rm Id}_V$ for all $x\in\LL$. In other words, the action of $\LL$ on $V$ extends to that of the associative algebra $U_\chi(\LL)$.

 We denote by $\mathcal{I}_\chi$ the ideal of the symmetric algebra $S(\LL)$ generated by all elements $(x-\chi(x))^p$ with $x\in \LL$. Since $\mathcal{I}_\chi$ is a Poisson ideal of the Lie--Poisson algebra $S(\LL)$, the factor-algebra $S_\chi(\LL):=S(\LL)/\mathcal{I}_\chi$ caries a Poisson algebra structure induced by that of $S(\LL)$. We call
$S_\chi(\LL)$ the {\it reduced symmetric algebra} associted with $\chi$. It is well--known that both $U_\chi(\LL)$ and $S_\chi(\LL)$ have dimension $p^{\dim\LL}$ over $\k$. The Lie algebra $\LL$ acts on $U_\chi(\LL)$ and $S_\chi(\LL)$ by derivations turning both algebras into {\it restricted} $\LL$-modules. The module structures thus obtained are induced by the adjoint action of $\LL$ on itself, and we shall use the subscript ``ad'' to indicate the Lie algebra actions.
In general, it is unknown for which restricted Lie algebras the modules $U_\chi(\LL)_{\rm ad}$ and $S_\chi(\LL)_{\rm ad}$ are isomorphic.
Answering this question would be very important from the point of view of the theory of support varieties developed by Friedlander and Parshall \cite{FP1, FP2} in the context of $U_\chi(\LL)$-modules.

Let $\N_p(\LL)=\{x\in\LL\,|\,\,x^{[p]}=0\}$, the  {\it restricted nullcone} of $\LL$. The $\k$-subalgebra
of $U_\chi(\LL)$ generated by a nonzero $x\in\LL$ will be denoted by $u_\chi(x)$. It follows from the PBW theorem that for $x\in\N_p(\LL)\setminus\{0\}$ the map $ X\mapsto x-\chi(x)$ gives rise to an algebra isomorphism $u_\chi(x)\cong \k[X](X^p)$. Given a finite dimensional $U_\chi(\LL)$-module 
$M$ we write $\mathcal{V}_{\LL}(M)^\times$ for the set all {\it nonzero} $x\in \N_p(\LL)$ such that $M$ is not a free $u_\chi(x)$-module, and put
$\mathcal{V}_{\LL}(M)\,:=\,\mathcal{V}_{\LL}(M)^\times\cup\{0\}$. It is easy to see that
$\mathcal{V}_{\LL}(M)$ is a Zariski closed, conical subset of $\N_p(\LL)$. It is called the {\it rank variety} (or the {\it support variety}) of $M$. It is proved in \cite{FP2} that $M$ is a projective
$U_\chi(\LL)$-module if and only if $\mathcal{V}_{\LL}(M)=\{0\}$.

Let $E_1,\ldots, E_s$ be representatives of the equivalence classes of all irreducible $U_\chi(\LL)$-module, and set $\mathcal{V}_{\LL}(\chi)\,:=\,\mathcal{V}_{\LL}(E_1)\cup\cdots\cup
\mathcal{V}_{\LL}(E_s).$ This Zariski closed, conical subset of $\N_p(\LL)$ contains the rank varieties of all finite dimensional $U_\chi(\LL)$-modules. By an important result of Jantzen \cite{Jan86}, we have that $\mathcal{V}_{\LL}(0)=\N_p(\LL)$, and it is proved in \cite{P99} that
$\mathcal{V}_{\LL}(\chi)=\mathcal{V}_{\LL}\big(U_\chi(\LL)_{\rm ad}\big)$ for any $\chi\in\LL^*$. 

From the representation-theoretic viewpoint it would be very useful to have a more explicit description of the variety
$\mathcal{V}_{\LL}(\chi)$. Let $\LL_\chi=\{x\in\LL\,|\,\,\chi([x,\LL])=0\} $, the stabiliser of $\chi$ in $\LL$. This is a restricted Lie subalgebra of $\LL$, and it is proved in \cite[Theorem~6.1]{PSk} that $\mathcal{V}_{\LL}\big(S_\chi(\LL)_{\rm ad}\big)=\,\N_p(\LL_\chi)$.
Therefore, if we know that $U_\chi(\LL)_{\rm ad}\cong S_\chi(\LL)_{\ad}$ as $\LL$-modules, then
$\mathcal{V}_{\LL}(\chi)\,=\,\N_p(\LL_\chi)$. 

According to \cite{P99}, there is a very large class of restricted Lie algebras for which the 
$\LL$-modules $U_\chi(\LL)_{\rm ad}$ and $S_\chi(\LL)_{\ad}$ are isomorphic. It includes the 
restricted Lie algebras $\LL$ of endomorphisms of a finite dimensional vector space $V$ (over $\k$) admitting direct complements in $\gl(V)$ invariant under the adjoint action of $\LL$ on $\gl(V)$ as well as the so-called
{\it saturated} Lie subalgebras of such Lie algebras; see \cite[\S 3]{P99} for relevant definitions and more detail. 
\subsection{The statement of the main result} A restricted Lie algebra is said to possess the {\it Richardson property} if there exists
a finite dimensional faithful restricted $\LL$-module $V$ with associated representation $\rho\colon\,\LL\to\gl(V)$ such that $\gl(V)=\rho(\LL)\oplus R$ where $R$ is a subspace of $\gl(V)$ such that $[\rho(\LL),R]\subseteq R$. According to  
\cite[Prop.~2.3]{P99}, if $\LL$ has the Richardson property then for any $\chi\in\LL^*$ the  $\LL$-modules $S_\chi(\LL)_{\rm ad}$ and $U_\chi(\LL)_{\rm ad}$ are isomorphic.

In \cite[3.6]{P99}, the author speculated that if $\LL$ possesses the Richardson property with respect to an irreducible faithful restricted $\LL$-module, then there exists a reductive algebraic group $G$ over $\k$ such that $\LL\cong\Lie(G)$ as restricted Lie algebras. This conjecture  was recently mentioned in the survey article \cite{BF} by Benkart--Feldvoss as one of the interesting open problems in the theory of modular Lie algebras; see Problem~7(c) in {\it loc.\,cit.} 

Let $G$ be a connected reductive $\k$-group. Recall that $p={\rm char}(\k)$ is a good prime for $G$ if all coefficients of the positive roots of $G$ expressed via a basis of simple roots are less than $p$. If $p$ is not good for $G$ then we say that $p$ is {\it bad}. The bad primes of connected reductive groups lie in the set $\{2,3,5\}$ and $p=5$ is bad for $G$ if and only if $G$ has a component of type ${\rm E}_8$.  A good prime $p$ is called {\it very good} if $G$ has no components of type ${\rm A}_{kp-1}$ with $k\ge 1$.
We say that $G$ is {\it standard} if $p$ is a good prime for $G$, the derived subgroup of $G$ is simply connected, and the Lie algebra $\g={\rm Lie}(G)$ admits a non-degenerate $(\Ad G)$-invariant symmetric bilinear form. A rational representation $\varphi\colon\, G\to \GL(V)$
is called {\it infinitesimally irreducible} if the differential ${\rm d}_e\varphi\colon\g\to\gl(V)$
is an irreducible representation of the Lie algebra $\g$.

The main goal of this note is to confirm the above-mentioned conjecture under the assumption that $p>3$. More precisely, we prove the following:
\begin{Theorem}\label{main-thm} Suppose ${\rm char}(\k)=p>3$ and
let $\LL$ be a restricted Lie algebra over $\k$ possessing the Richardson property with respect to a finite dimensional faithful irreducible restricted representation $\rho\colon\, \LL\to \gl(V)$. 
Then there exists a standard reductive algebraic $\k$-group $G$ with Lie algebra $\g$ and an infinitesimally irreducible rational representation 
$\varphi\colon\,G\to\GL(V)$ such that
$({\rm d}_e\varphi)(\g)=\rho(\LL)$. Moreover,
$\LL$ and $\g$ are isomorphic as restricted Lie algebras.  
\end{Theorem} By Schur's lemma, if $\LL$ admits a faithful irreducible representation then the centre of $\LL$ has dimension $\le 1$. 
In the final part of this note we  
show that if $p$ is a very good prime for a standard algebraic group $G$ whose centre has dimension $\le 1$, then the restricted Lie algebra $\g=\Lie(G)$ possesses the Richardson property with respect to a faithful irreducible representation $\rho\colon\,\g\to\gl(V)$ such that $p\nmid \dim V$. Furthermore, $\rho={\rm d}_e\varphi$, where $\varphi\colon\,G\to\GL(V)$
is an infinitesimally irreducible rational representation, and the 
$(\Ad G)$-invariant
trace form $(X,Y)\mapsto {\rm tr}\big(\rho(X)\circ \rho(Y)\big)$ on $\g$ is non-degenerate. 


In characteristics $2$ and $3$, we show that if $\LL$ has the Richardson property with respect to a finite dimensional faithful restricted $\LL$-module $V$, then for any nonzero $e\in\N(\LL)$ there exists an element $h\in \LL_{ss}$ such that $[h,e]=e$. In conjunction with some results of \cite{P87b} and \cite{P90} mentioned in Subsection~\ref{restricted} this implies that all Cartan subalgebras of $\LL$ are toral and have the same dimension equal to  ${\rm MT}(\LL)$.
In characteristics $3$, we show that any $e\in\N(\LL)$ lies in the subspace $[e,[e,\LL]]$. This enables us to deduce that the solvable radical of any subalgebra $\LL(\gamma_1,\ldots, \gamma_d)$ is toral and coincides with the centre of  $\LL(\gamma_1,\ldots, \gamma_d)$. We then use Skryabin's classification 
\cite{Sk} of simple $3$-modular Lie algebras of toral rank $1$ to show that for every $\gamma\in\Gamma(\LL,\t)$ the radical of $\LL(\gamma)$ coincides with $\ker\gamma\,=\,\{t\in\t\,|\,\, \gamma(t)=0\}$ and $\LL[\gamma]=\LL(\gamma)/\ker\gamma$ is one of $\sl_2$ or $\psl_3$.
It would be interesting to classify the finite dimensional restricted Lie algebras over fields of characteristic $3$ having these properties.
\subsection{Lie algebras without strong degeneration}\label{strong}
The key idea of the proof of Theorem~\ref{main-thm} is to show that if 
$\LL$ has the Richardson property then the factor-algebra $\LL/\z(\LL)$ does not contain nonzero elements $c$ with $(\ad\,c)^2=0$ and then apply the classification of such Lie algebras obtained in \cite{P86, P87a}. This is the main reason for us to impose the assumption that $p>3$.

An element $c\in\LL$ is called a {\it sandwich element} if $(\ad c)^2=0$. This term, coined by A.I.~Kostrikin, has to do the fact that in characteristic $p>2$ any $c\in\LL$ with $(\ad c)^2=0$ satisfies the {\it sandwich identity} $$(\ad c)\circ (\ad x)\circ (\ad c)=0\qquad\quad(\forall\,x\in\LL).$$ It is immediate from this identity that the set  $\mathfrak{C}(\LL)$ of all sandwich elements of $\LL$ is closed under taking Lie brackets.
In conjunction with the Engel--Jacobson theorem on weakly closed sets this shows that $\mathfrak{C}(\LL)$ generates a nilpotent Lie subalgebra of $\LL$ invariant under the automorphism group of $\LL$.
We say that $\LL$ is {\it strongly degenerate} if
$\mathfrak{C}(\LL)\ne \{0\}$.

Sandwich elements play an important role in the study of Engel Lie algebras and in the classification theory of finite dimensional simple Lie algebras over algebraically closed fields of 
characteristic $p>3$; see \cite{K90}, \cite{KZ90}, \cite{P94}, \cite{PS1}, \cite{Str17}, \cite{Z91}.
 In his talks at the ICM's in Stockholm (1962) and Nice (1971), Kostrikin conjectured that
 a finite dimensional simple Lie algebra $L$ over an algebraically closed field of characteristic $p>3$ is either strongly degenerate or classical in the sense of Seligman, that is, has the form $L=[\Lie(G),\Lie(G)]$ for some simple algebraic $\k$-group $G$ of adjoint type. Kostrikin's conjecture attracted a lot of attention 
 in 1960's and 1970's and was first confirmed under various additional assumptions on $L$; see \cite{Ko67}, \cite{Ja71}, \cite{St73} and \cite{Be77}\footnote{This work of Georgia Benkart is based on her Yale PhD thesis \cite{Be74} written under the supervision of Professor Nathan Jacobson.}. In full generality, the conjecture was proved in \cite{P87a} for $p>5$ and in \cite{P86} for $p=5$. 

In fact, a slightly more general result was proved in \cite{P87a, P86}.
A finite dimensional Lie algebra $L$ over an algebraically closed field of characteristic $p>3$ is
called {\it almost classical} if there exists
a semisimple  algebraic $\k$-group $G$ of adjoint type such that $L$ is isomorphic to a subalgebra of $\Lie(G)$ containing  $[\Lie(G),\Lie(G)]$. Note that under our assumptions on $p$  the Lie algebra $\Lie(G)$ identifies with the derivation algebra of the restricted Lie algebra
$[\Lie(G),\Lie(G)]$. The latter decomposes into a direct sum of classical simple Lie algebras which may include components isomorphic to $\psl_{rp}(\k)$ with $r>0$.
The main result of \cite{P86, P87a} states that for $p>3$ a finite dimensional Lie algebra $L$ over $\k$ is almost classical if and only if $\mathfrak{C}(L)=\{0\}$. 

{\bf Acknowledgement.} This work was started during the author's stay at MSRI (Berkeley) 
in February--April 2018. I would like to thank the Mathematical Sciences Research Institute for its hospitality and creative atmosphere during the programme ``Group Representation Theory and Applications''.
\section{The case of arbitrary Richardson $\LL$-modules}\label{Prelim}
\subsection{}\label{intro1}
From now on all $\LL$-modules are assumed to be finite dimensional and restricted. We say that a faithful $\LL$-module $V$ is {\it Richardson} if there exists a subspace $R$ of $\gl(V)$
such that $[\LL,R]\subseteq R$ and $\gl(V)=\LL\oplus R$, where we identify $\LL$ with its image in $\gl(V)$. 
Given a module $M$ for a Lie algebra $\g$ and $v\in M$ we denote by $\g_v$ the stabiliser of $v$ in $\g$. We write $\z(\g)$ for the centre of $\g$ and ${\rm rad}\, \g$ for the {\it radical} of $\g$, the largest solvable ideal of $\g$.

The following
lemma is inspired by a very old observation of Jacobson.
\begin{Lemma}\label{L1}
	Let $V$ be a Richardson module for $\LL$ (not necessarily irreducible).
	\begin{itemize}
		\item[(1)]\, If $p>2$ then any $e\in\N(\LL)$ lies in the subspace $[e,[e,\LL]]$.
		
		\smallskip
		\item[(2)]\, If $p>0$ then for any $e\in\N(\LL)$ there is an element $h\in \LL_{ss}$ such that $[h,e]=e$. 
		\end{itemize}
	\end{Lemma}
\begin{proof}
	Thanks to our conventions we may assume that 
	$e$ is a nonzero nilpotent element of $\gl(V)$. We first look for $x,y\in \gl(V)$ such that 
	$e=[e,[e,x]]$ and $e=[y,e]$. For $p>2$ one can argue as in \cite[Lemma~2]{Jac} to observe that $e$ can be included into an $\sl_2$-triple $\{e,h,f\}\subset\gl(V)$. For $p>0$, the Lie algebra $\gl(V)$ admits
	a $\Z$-grading $\gl(V)\,=\,\bigoplus_{i\in\Z}\,\gl(V)_i$ such that
	 $e\in \gl(V)_2$ and $[e,\gl(V)_0]=\gl(V)_2$; see \cite[Theorem~A(ii)]{P03a}. This result is applicable in arbitrary characteristic since $\gl(V)$ admits a non-degenerate $(\Ad \GL(V))$-invariant trace form. 
	 
	 As a consequence, there is $y\in \gl(V)$ such that $[y,e]=e$. Write $y=y'+y''$ with
	 $y'\in \LL$ and $y''\in R$. As $e\in\LL$ and $[\LL,R]\subseteq R$ it must be that $[y'',e]=0$ and $[y',e]=e$. As $\LL$ is a restricted Lie subalgebra of $\gl(V)$ we can replace $y'$
	 by its semisimple part $y'_s\in\la y'\ra_p$ to find a semisimple element $h\in \LL$ such that $[h,e]=e$. This proves (2).
	 
	 If $p>2$ then $[e,[e,f]]=-2e\in\k^\times e$. Write $f=f'+f''$ with $f'\in\LL$ and $f''\in R$.
	 Since $(\ad e)^2$ preserves both $\LL$ and $R$ it must be that $[e,[e,f']]\in\k^\times e$. This proves (1). 
	 \end{proof}
 \subsection{}\label{S0}
	 There are examples of restricted Lie algebras admitting infinite families of indecomposable Richardson modules.
	 Indeed, suppose $\LL=\sl_2$ and $p>3$. Let $V(m)$ be the Weyl module for the $\k$-group $\SL(2)$ with highest weight $m\in\Z_{\ge 0}$. Differentiating the rational action of $\SL(2)$ endows $V(m)$ with a natural restricted $\LL$-module structure. 
	 According to \cite{P91} the restricted $\LL$-module $V(m)$ is indecomposable if and only if $p\nmid (m+1)$. It is well--known that the $\LL$-module $V(m)$ is simple if and only if $m\le p-1$. Furthermore, the Steinberg module ${\rm St}=V(p-1)\cong V(p-1)^*$ is simple and projective over the restricted enveloping algebra $U_0(\LL)$. 
	 
	 Suppose $m=kp+l$ where $k\in \Z_{\ge 0}$ and $1\le l \le (p-3)/2$. Then 
	 \cite[Lemma~2.6(i)]{P91} shows that $$V(m)\otimes V(m)^*\cong\,P\oplus \bigoplus_{i=0}^{l}V(2i)$$ 
	 where $P$ is a projective $U_0(\LL)$-module isomorphic to a direct sum of $k$ copies of 
	 ${\rm St}\otimes V(kp+2l+1)$. The Weyl module $V(kp+2l+1)$ has two composition factors
	 $V(2l+1)$ and $V(p-3-2l)$, while the composition factors of the projective $U_0(\LL)$-module ${\rm St}\otimes V(2)$ are $V(p-3)$, $V(1)$ and $V(p-1)$; see \cite[Prop.~1.6 and Theorem~1.11(c)]{BO}.
	 Since $\{2l+1,p-3-2l\}\cap \{1,p-3,p-1\}=\varnothing$ we have that
	 $${\rm Hom}_{\LL}\big(V(2), {\rm St}\otimes V(kp+2l+1)\big)\cong\, 
	 {\rm Hom}_{\LL}\big(V(2)\otimes{\rm St},V(kp+2l+1\big)\,=\,0,$$
	 implying that ${\rm Hom}_{\LL}(V(2), P)=0$. 
	 Since
	 $V(m)\otimes V(m)^*\cong \,\gl(V(m))$ and $V(2)\cong \LL$ as $\LL$-modules, the above yields that  the $\LL$-module $\gl(V(m))$ contains a direct summand $R$ of codimension $3$ such that
	 ${\rm Hom}_{\LL}(\LL,R)=0$. Identifying the Lie algebra $\LL$ with its copy in $\gl(V(m))$ we now get $\LL\cap R=\{0\}$ and hence $\gl(V(m))\,=\,\LL\oplus R$. As a result, each
	 $V(m)$ with $1\le l\le (p-3)/2$ and $k\in\Z_{\ge 0}$ is an indecomposable Richardson module for $\LL$.  
	 
	 If $\LL=\sl_2$ and $p=3$ then $\LL\cong V(p-1)$ is a projective module over $U_0(\LL)$. This means that {\it any} non-trivial finite dimensional restricted $\LL$-mudule is Richardson for $\LL$.
	 Such instances are, of course, extremely rare.
\subsection{} Suppose $\LL$ admits a Richardson module $V$ (not necessarily irreducible). Our next result provides some insight into the structure of $\LL$.
	 \begin{Theorem}\label{T1}
	 	The following are true:
	 	\begin{itemize}
	 	\item[(1)] If $p\ge 2$ the all Cartan subalgebras of $\LL$ are toral of dimension equal to ${\rm MT}(\LL)$.
	 	
	 	\smallskip
	 	
	 	\item[(2)] Suppose $p>2$ and let $\t$ be any toral subalgebra of $\LL$. Then ${\rm rad}\,\c_{\LL}(\t)\,=\,\z(\c_{\LL}(\t))$ is a toral subalgebra of $\LL$ and the factor-algebra
	 	$\c_{\LL}(\t)/\z(\c_{\LL}(\t))$ has no nonzero sandwich elements.
	 	
	 	\smallskip
	 	
	 	\item[(3)] If $p>3$ then the factor-algebra $\LL/\z(\LL)$ is almost classical.	
	 		\end{itemize}
	 	\end{Theorem}
	\begin{proof}
		Let $\h$ be a Cartan subalgebra of $\LL$. It is well--known (and easy to see) that 
		$\h=\c_{\LL}(\t')$ where $\t'=\h\cap\LL_{ss}$, a maximal toral subalgebra of $\LL$. Let $\Gamma$ be the set
		of roots of $\LL$ with respect to $\t'$ and 
		$\widetilde{R}=R\oplus \sum_{\gamma\in\Gamma}\,\LL_\gamma$. Then $\gl(V)=\h\oplus \widetilde{R}$ and $[\h,\widetilde{R}]\subseteq \widetilde{R}$, meaning that $V$ is a Richardson module for $\h$. If $\h$ contains a nonzero nilpotent element, $e$ say, then  Lemma~\ref{L1}(2) yields that there is an  $h\in \h$ such that $[h,e]=e$. But then $(\ad h)^n(e)=e\ne 0$ for all $n\in\Z_{>0}$. Since $\h$ is nilpotent this is impossible. Since $\h$ is a restricted subalgebra of $\LL$, the semisimple and nilpotent parts of any element of $\h$ lie in $\h$. Consequently,  $\h=\h_{ss}=\t'$. 
		
		Now let $\t$ be any torus of maximal dimension in $\LL$. Then
		$\dim \t\ge \dim \t'$ and $\dim \c_{\LL}(\t)\le\dim \h$ by the main result of \cite{P87b} 
		(see Subsection~\ref{restricted} for a related discussion). Since $\h=\t'$ we now deduce that
		$\dim\h=\dim \t'={\rm MT}(\LL)$, proving (1).
		
		Suppose $p>2$ and let $\t$ be a toral subalgebra of $\LL$ (possibly zero). Let $\LL_0=\c_{\LL}(\t)$ and write $\Gamma(\LL,\t)$ for the set of roots of $\LL$ with respect to $\t$ (possibly empty). Now set $\widetilde{R}:=R\oplus \sum_{\gamma\in\Gamma(\LL,\t)}\,L_\gamma$. Then $\gl(V)=\LL_0\oplus\widetilde{R}$ and $[\LL_0, \widetilde{R}]\subseteq \widetilde{R}$, showing that $V$ is a Richardson module for $\LL_0$. Let $\z$ be the centre of $\LL_0$, a restricted ideal of $\LL_0$. If $\z$ contains a nonzero nilpotent element, $e$ say, then
		$e\in [e,[e,\LL_0]]$ by Lemma~\ref{L1}(1). But then $0\ne e\in [e,\z]$, a contradiction. 
		Hence $\z$ is toral. 
		
		Suppose $\LL_0/\z$ contains a nonzero sandwich element. Then there is $c\in \LL_0\setminus\z$ such that $[c,[c,\LL_0]]\subseteq\z$. Since $p\ge 3$ it must be that  $c^{[p]}\in\z$. As $\z$ is toral, we may replace $c$ by a suitable element of the form $c+z$ with $z\in\z$ to assume further that $c$ is nilpotent. By Lemma~\ref{L1}(1), we then have $c\in[c,[c,\LL_0]]\subseteq \z$, a contradiction. Hence $\LL_0/\z$ has no nonzero
		sandwich elements. In particular, $\LL_0/\z$ has no nonzero abelian ideals.
		As a consequence, the radical of $\LL_0/\z$ is trivial.
		
		Let $\mathfrak{r}={\rm rad}\,\LL_0$. Then $\z\subseteq \mathfrak{r}$ and $\mathfrak{r}/\z$ is a solvable ideal of $\LL_0/\z$. The preceding remark shows that 
		$\mathfrak{r}=\z$, proving (2). Part (3) now follows from (2) and the main results of \cite{P87a} and \cite{P86} (see our discussion in Subsection~\ref{strong} for more detail).
	\end{proof}
In characteristic $3$, one can use Skryabin's description of simple Lie lagebras of rank $1$ to get more information on the $1$-sections of $\LL$.
\begin{Corollary}
Suppose $p=3$ and let $\alpha\in\Gamma(\LL,\t)$ where $\t$ is a maximal torus of $\LL$. Then
${\rm rad}\,\LL(\alpha)\,=\,\t(\alpha)\,=\,\ker\alpha$ and $\LL[\alpha]=\LL(\alpha)/\t(\alpha)$ is one of $\sl_2$ or
$\psl_3$.
\end{Corollary}
\begin{proof} Since $p=3$, Theorem~\ref{T1}(1) yields that $\LL(\alpha)\,=\,\LL_{-\alpha}\oplus \t\oplus\LL_\alpha$. 
Since $\t$ is a maximal torus and ${\rm rad}\,\LL(\alpha)\,=\,\z(\LL(\alpha))$ is toral by Theorem~\ref{T1}(2), it must be that ${\rm rad}\,\LL(\alpha)\subseteq \t$. Since $\LL_\alpha$ is nonzero, $\t(\alpha)\subseteq \z(\LL(\alpha))$ and $\LL[\alpha]$ is semisimple by Theorem~\ref{T1}(2), we have that $\t(\alpha)\,=\,{\rm rad}\,\LL(\alpha)$ and
${\rm MT}(\LL[\alpha])=1$. Furthermore, $\LL_{-\alpha}\ne \{0\}$ as otherwise $\LL(\alpha)$ would be solvable. 

Since ${\rm MT}(\LL[\alpha])=1$ the Lie algebra $\LL[\alpha]$ has a unique minimal ideal, say $S$.
Since $\LL[\alpha]$ has no nonzero sandwich elements by Theorem~\ref{T1}(2), repeating verbatim the argument used in the proof of \cite[Lemma~4]{P87a} one observes that the ideal $S$ is simple and $\LL[\alpha]$ is sandwiched between $\ad S$ and ${\rm Der}\,S$. Since ${\rm MT}(\LL[\alpha])=1$, the simple Lie algebra $S$ has absolute toral rank $1$. Applying \cite[Theorem~6.5]{Sk} one obtains that either $S\cong\sl_2$ or $S\cong \psl_3$. 

In any event, $S$ is a restricted ideal of $\LL[\alpha]$. Let $\widetilde{S}$ be the inverse image of $S$ in $\LL(\alpha)$, a restricted ideal of $\LL(\alpha)$, and let $\t'$ be a maximal torus of 
$\widetilde{S}$. Since $\t'$ contains $\t(\alpha)$ it is straightforward to see that $\dim \t'=\dim\t$. Applying Theorem~\ref{T1}(1) once again yields that $\t'$ is a self-centralising maximal torus of $\LL(\alpha)$. As $[\t',\LL(\alpha)]\subseteq \widetilde{S}$, all root subspaces of $\LL(\alpha)$ with respect to $\t'$ are contained in $\widetilde{S}$. But then $\LL(\alpha)\subseteq  \t'+\widetilde{S}$ forcing $\LL(\alpha)=\widetilde{S}$. This completes the proof.
\end{proof}
\section{The case of irreducible Richardson modules}
\subsection{}\label{3.1} From now on we assume that $p>3$ and our Richardson $\LL$-module $V$ is irreducible. In particular, this means that $\c_{\gl(V)}(\LL)$ consists of scalar endomorphisms of $V$ (by Schur's lemma). Therefore, either $\c_{\gl(V)}(\LL)\,=\,\z(\LL)$ or $\LL$ is centreless and $\c_{\gl(V)}(\LL)\subset R$. By Theorem~\ref{T1}(3) (and our discussion in Subsection~\ref{strong}) there exists a semisimple algebraic $\k$ group $G$ of adjoint type such that
$\LL/\z(\LL)$ identifies with a Lie subalgebra of $\g:=\Lie(G)$ containing $[\g,\g]$.  Since $G$ is a group of adjoint type and $p>3$, it follows from \cite[Lemma~2.7]{BGP}, for example, that the restricted Lie algebra $\g$ is isomorphic to $\Der [\g,\g]$. 
Therefore, we may identify $\bar{\LL}:=\LL/\z(\LL)$ with a restricted ideal of $\g$ containing $[\g,\g]$.  

Let $T$ be a maximal torus of $G$ and denote by $\Phi$ the root system of $G$ with respect to $T$.
Let $\Pi$ be a set of simple roots in $\Phi$ and write $\Phi_+=\Phi_+(\Pi)$ for the set of positive roots with respect to $\Pi$.
All root subspaces $\g_\alpha$ with $\alpha\in\Phi$ are contained in $[\g,\g]$ and
$\bar{\LL}\,=\,\bar{\t}\oplus \sum_{\alpha\in\Phi}\,\g_\alpha$ where $\bar{\t}:=\bar{\LL}\cap \Lie(T)$.
Since $[\g_\alpha,\g_{-\alpha}]\subset \bar{\t}$ for all $\alpha\in\Phi$ and $p>3$, it follows from Seligman's results that $\bar{\t}$ is a self-centralising maximal torus of $\bar{\LL}$ and for every $\gamma\in \Gamma(\bar{\LL},\bar{\t})$ there is a unique $\widehat{\gamma}\in\Phi$ such that $\gamma=({\rm d}_e\widehat{\gamma})\vert_{\,\bar{\t}}$; see \cite[Ch.~II, \S 3]{Sel}. 

Let $\t$ be the inverse image of $\bar{\t}$ in $\LL$, By Theorem~\ref{T1}(1), this is a toral Cartan subalgebra of $\LL$. The above discussion enables us to identify $\Phi$ with  the set of roots $\Gamma(\LL,\t)$, i.e. given $x\in\t$ we write $\alpha(x)$ instead of $({\rm d}_e\alpha)(\bar{x})$ with $\bar{x}=x+\z(\LL)$. We have $\LL\,=\,\t\oplus \sum_{\alpha\in\Phi}\,\LL_\alpha$ and each root subspace $\LL_\alpha=\k e_\alpha$ with respect to $\t$ is $1$-dimensional. We choose root vectors $e_\alpha$ in such a way that $[[e_\alpha,e_{-\alpha}],e_\alpha]=2e_\alpha$ for all $\alpha\in\Phi_+$
and embed the torus $T$ of $G$ into ${\rm Aut}(\LL)$ by setting $t(x)=x$ and $t(e_\alpha)=\alpha(t)e_\alpha$ for all $t\in T$, $x\in\t$ and $\alpha\in\Phi$. 
 \subsection{}\label{3.2}  
 Given $\alpha\in\Phi_+$ we put $h_\alpha:=[e_\alpha,e_{-\alpha}]$. Obviously, $\{e_\alpha,h_\alpha,e_{-\alpha}\}$ is an $\sl_2$-triple in $\LL$ and we write $\mathfrak{s}_\alpha$ for the $\k$-span of $e_{\pm \alpha}$ and $h_\alpha$. It is straightforward to see that $\LL$ admits a natural restricted Lie algebra structure $x\mapsto 
 x^{[p']}$ with respect to which all $h_\alpha$ are toral and all $e_\alpha$ have the property that $e_\alpha^{[p']}=0$. Unfortunately, the restricted Lie algebras $(\LL,[p])$ and
 $(\LL,[p'])$ may be very different if $\dim\z(\LL)=1$. In order to proceed further we have to investigate this problem. 
 
 Suppose that $\z(\LL)$ is spanned by $z={\rm Id}_V$.
 Since $\ad(x^{[p']})=\ad(x^{[p]})=(\ad x)^p$, there exists a linear function $\chi\in\LL^*$ such that 
 $x^{[p']}=x^{[p]}+\chi(x)^p\cdot z$ for all $x\in\LL$. Since $V$ is a restricted module for $(\LL,[p])$, it has $p$-character $\chi$ when regarded as a module over $(\LL,[p'])$.
   \begin{Lemma}\label{L2}
   Suppose $\dim\z(\LL) =1$. Then the linear function $\chi$ vanishes on $[\LL,\LL]$.	
   \end{Lemma}
   \begin{proof}
   	Suppose for a contradiction that $\chi$ does not vanish on $\mathfrak{s}_\beta$ for some $\beta\in\Phi_+$. The above discussion enables us to view $V$ as a $U_\chi(\mathfrak{s}_\beta)$-module. (To ease notation we identify $\chi$ with its restriction to $\mathfrak{s}_\beta$.) By the choice of $\beta$, the stabiliser of $\chi$ in $\mathfrak{s}_\beta$ is a {\it proper} Lie subalgebra of $\mathfrak{s}_\beta$, hence $1$-dimensional. Since the restricted Lie algebra $\sl_2$ has the Richardson property,
   	our discussion in Subsection~\ref{supp} shows that $\mathcal{V}_{\mathfrak{s}_\beta}(\chi)\,=\,\N_p((\mathfrak{s_\beta})_\chi)$ is either zero or a single line depending on $\chi$. 
   	
   	We may thus assume without loss of generality that $e_\beta\not\in  \mathcal{V}_{\mathfrak{s}_\beta}(\chi)$. Then all composition factors of the $u_\chi(\mathfrak{s}_\beta)$-module $V$ are free $u_\chi(e_\beta)$-modules, implying that all Jordan blocks of $e_\beta$ on $\gl(V)\cong V\otimes V^*$ have size $p$. Since $\LL$ is a direct summand of $\gl(V)$, all Jordan blocks of $\ad e_\beta\in\gl(\LL)$ must have size $p$ as well. In particular, $(\ad e_\beta)^{p-1}\ne 0$. 
   	
   	Now, if $\gamma\in\Phi\setminus\{\pm \alpha\}$ then $(\ad e_\beta)^{p-1}(e_\gamma)=0$ by the Mills--Seligman axioms 
   	(see \cite[Theorem~2.5]{BGP}, for example). Since $(\ad e_\beta)^3(\mathfrak{s}_\beta)=0=(\ad e_\beta)^2(\t)$ and $p>3$ we get $(\ad e_\beta)^{p-1}=0$. This contradiction shows that $\chi$ vanishes on all subalgebras $\mathfrak{s}_\beta$ with $\beta\in \Phi_+$. As such subalgebras span $[\LL,\LL]$, the result follows.
   	\end{proof}
   \subsection{}\label{3.3} It follows from Lemma~\ref{L2} that $e_{\pm \alpha}^{[p]}=0$ and $h_\alpha^{[p]}=h_\alpha$ for all $\alpha\in\Phi_+$. This must hold in the case where 
   $\LL$ is centreless as well, since $\LL$ then admits a unique $[p]$-structure. Now one can check 
   directly the action of the maximal torus $T$ of $G$ on $\LL$ described in Subsection~\ref{3.1} respects the $[p]$-structure of $\LL$, i.e. has the property that $t(x^{[p]})=t(x)^{[p]}$ for all $t\in T$ and $x\in \LL$ (one has to keep in mind here that $\t$ is restricted and $T$ acts trivially on $\t$). As a consequence, $T$ embeds into the automorphism group of the restricted enveloping algebra $U_0(\LL)$.
   \begin{Lemma}\label{L3}
 The module $V$ admits a rational $T$-action compatible with the action of $T$ on $\LL$.  
\end{Lemma}
  \begin{proof}
  By \cite[Prop.~3.5]{GG82} and \cite[Corollary~2]{Jan00}, the simple $U_0(\LL)$-module $V$ is gradable, i.e. decomposes into a direct sum of $T$-weight spaces $V=\bigoplus _{\lambda\in X(T)} V_\mu$ in such a way that each weight space $V_\mu$ is $\t$-stable and $t(e_\alpha . v)= (\mu+\alpha)(t)v$ for all $t\in T$,  $\alpha\in\Phi$ and $v\in V_\mu$.
  \end{proof}
Since $G$ is a group of adjoint type the torus $T$ acts faithfully on $[\g,\g]$. As $\LL$ embeds into $\gl(V)$, this implies that $T$ acts faithfully on $V$. We may thus identify $T$ with a connected subgroup of $\GL(V)$. Since $T\subset N_{\GL(V)}(\LL)$ the Lie algebra $\Lie(T)$ normalises $\LL$. 

Recall from Subsection~\ref{3.1} that $\bar{\LL}=\LL/\z(\LL)$ is sandwiched between $\g$ and $[\g,\g]$.
\begin{Proposition}\label{P1}
We have the equality $\bar{\LL}=\g$.	
	\end{Proposition}
\begin{proof}
	Let $\n(\LL)$ denote the normaliser of $\LL$ in $\gl(V)$. Given $x\in \n(\LL)$ we write $x=x'+x''$ with
	$x'\in \LL$ and $x''\in R$. As $[x,\LL]\subseteq \LL$ it must be that
	$[x'',\LL]=0$. Hence $\n(\LL)\,=\,\LL\oplus (R\cap\c_{\gl(V)}(\LL))$. Since $\Lie(T)\subset 
	\n(\LL)$ acts faithfully $[\g,\g]$ (and hence on $\LL$) the restriction of the first projection 
	$\n(\LL)\to \LL$ to $\Lie(T)$ is injective. Since $\g=\Lie(T)+\sum_{\alpha\in\Phi}\,\bar{\LL}_\alpha$ this implies that $\dim\,\g=\dim\,\bar{\LL}$. As $\bar{\LL}\subseteq \g$, the claim follows.
\end{proof}
\subsection{}\label{3.4} Suppose $H$ is a simple algebraic $\k$-group of adjoint type and $p>3$. Then $\Lie(H)$ is a perfect Lie algebra whenever $H$ has type other than ${\rm A}_{kp-1}$, and  $[\Lie(H),\Lie(H)]$ has codimension $1$ in $\Lie(H)$ if $H\cong{\rm PGL}_{kp}$. Indeed, in the latter case $\Lie(H)\,\cong\,\mathfrak{pgl}_{kp}$ and $[\Lie(H),\Lie(H)]\,\cong\,\mathfrak{psl}_{kp}$. Since our algebraic group $G$ is isomorphic to a a direct product of simple algebraic groups of adjoint type, $\dim(\g/[\g,\g])$ equals the number of simple components of $G$ having types ${\rm A}_{kp-1}$ with $k\in\Z_{>0}$.
\begin{Lemma}\label{L4}
The derived subalgebra of $\g$ has codimension $\le 1$ in $\g$.	
\end{Lemma}
\begin{proof}
	In view of Proposition~\ref{P1}, in order to prove the lemma it suffices to show that $\LL$ has codimension $\le 1$ in $\LL$.
	Since $\LL$ is a direct summand of $\gl(V)$, a self-dual $\LL$-module, the coadjoint $\LL$-module $\LL^*$ must be a direct summand of $\gl(V)$ as well. Since the trivial $\LL$-submodule $(\LL^*)^\LL$ of $\LL^*$ identifies canonically with the dual space $(\LL/[\LL,\LL])^*$, we have that  
	$$\dim(\LL/[\LL,\LL])=\dim(\LL^*)^\LL\le \dim\c_{\gl(V)}(\LL)=1.$$
	If follows that either $\g$ is perfect or $[\g,\g]$ has codimension $1$ in $\g$. 
\end{proof}
	Lemma~\ref{L4} shows that the group $G$ cannot have more that one component of type ${\rm A}_{kp-1}$, and its proof implies that $[\LL,\LL]$ has codimension $\le 1$ in $\LL$.
	\begin{Lemma}\label{L5}
		If $G$ has a component of type ${\rm A}_{kp-1}$, then $\z(\LL)$ is a $1$-dimensional subalgebra of $[\LL,\LL]$. In other words, $[\LL,\LL]$ is a non-split central extension of  $[\g,\g]$.
		\end{Lemma}
	\begin{proof} Since $G$ has a component of type ${\rm A}_{kp-1}$ it follows from Lemma~\ref{L4} that $G\cong G_1\times G_2$ where $G_1$ has no components of type ${\rm A}_{kp-1}$ and $G_2\cong{\rm PGL}_{kp}$. Setting $\g_i:=\Lie(G_i)$ for $i=1,2$ we get two commuting restricted ideals of $\g$ such that $\g=\g_1\oplus\g_2$. Furthermore, $[\g_1,\g_1]=\g_1$ is centreless and $\g_2\cong\mathfrak{pgl}_{kp}$.
		If $\z(\LL)\cap [\LL,\LL]=\{0\}$ then $\LL=\z(\LL)\oplus [\LL,\LL]$. In this case, Proposition~\ref{P1} entails that $$\g=\bar{\LL}\,\cong\,[\LL,\LL]\,\cong\,\g_1\oplus[\g_2,\g_2]\,\cong\,\g_1\oplus\psl_{kp}$$
		is a completely reducible $\ad \g$-module. Since $\mathfrak{pgl}_{kp}$ is a direct summand
		of $\g$, this is impossible. By contradiction, the result follows.
		\end{proof}
	\subsection{}\label{3.5}  If $G\cong G_1\times G_2$, where the $G_i$'s are as in the proof of Lemma~\ref{L5}, we let $\widetilde{G} = \widetilde{G}_1\times \widetilde{G}_2$ be the reductive $\k$-group such that $\widetilde{G}_1$ is a semisimple, simply connected cover of $G_1$ and $\widetilde{G}_2=\GL_{kp}$.
	If $G$ has no components of type ${\rm A}_{kp-1}$ and $\LL$ is centreless, we set $\widetilde{G}:=\widetilde{G}_1$. Finally, if $G$ has no components of type ${\rm A}_{kp-1}$ and $\dim\z(\LL)=1$ , we set $\widetilde{G}:=T_0\times \widetilde{G}_1$ where $T_0$ is a $1$-dimensional central torus of $\widetilde{G}$. In all cases, the derived subgroup of $\widetilde{G}$ is semisimple and simply connected. Since $p>3$, the Lie algebra 
	$\g_1\cong \Lie(\widetilde{G}_1)$ is a direct sum of simple ideals. By \cite[Lemma~2.2]{Pr97}, the Lie algebra $\widetilde{\g}:=\Lie(\widetilde{G})$ admits a non-degenerate symmetric $(\Ad \widetilde{G})$-invariant bilinear form, 
	say $b\colon\, \widetilde{\g}\times \widetilde{\g}\to \k$, and it follows from
	\cite[Theorem~A]{Gar} that $b$ can be chosen to be a trace form (associated with a rational representation of $\widetilde{G}$) provided that $p$ is a good prime for
	$\widetilde{G}_1$.
	\begin{Proposition}\label{P2} 
	In all cases, the restricted Lie algebra $\LL$ is isomorphic to $\widetilde{\g}$.
	\end{Proposition}
\begin{proof}
	By construction, $\z(\widetilde{\g})$ lies in $[\widetilde{\g},\widetilde{\g}]$ and coincides with the radical of the restriction of
	$b$ to $[\widetilde{\g},\widetilde{\g}]$. Therefore, $b$ gives rise to a non-degenerate $\g$-invariant symmetric bilinear form on $[\g,\g]\,\cong\, [\widetilde{\g},\widetilde{\g}]/\z(\widetilde{\g})$; we call it $\bar{b}$. Since $\Der [\g,\g]\,\cong \g$ and $\g\cong \widetilde{\g}/\z(\widetilde{\g})$ as Lie algebras, the bilinear form $\bar{b}$ is invariant under all derivations of $[\g,\g]$. It follows that for any $2$-cocycle
	$\varphi$ on $[\g,\g]$ with values in $\k$ there is an element $h\in\g$ such that 
	$\varphi(x,y)=\bar{b}([h,x], y)$ for all $x,y\in[\g,\g]$. Furthermore, the central extension of $[\g,\g]$ associated with $\varphi$ is trivial if and only if $h\in [\g,\g]$.
	
	If $G=G_1$ then $\g$ is perfect, whilst if $G$ has a component of type ${\rm A}_{kp-1}$ then $[\g,\g]$ has codimension $1$ in $\g$. In case $G=G_1$, this shows that all central extentions of $\g=[\g,\g]$ are trivial. Hence $\LL\cong\g$ if 
	$\LL$ is centreless and $\LL\,\cong\,\k\oplus\g\,\cong\, \Lie(T_0\times \widetilde{G}_1)$
	if $\dim \z(\LL)=1$ (we view $\k$ is a $1$-dimensional toral Lie algebra). As $\z(\LL)$ is a torus, this is an isomorphism of restricted Lie algebras.
	
	If $G$ has a component of type ${\rm A}_{kp-1}$ then 
	$[\widetilde{\g},\widetilde{\g}]\,\cong\, \g_1\oplus \sl_{kp}$ is a non-split central extension of $[\g,\g]$. Since the above discussion implies that
	$[\g,\g]$ admits a unique non-trivial central extension (up to equivalence), it follows from Lemma~\ref{L5} that
	$[\LL,\LL]\cong\g_1\oplus \sl_{kp}$ as ordinary Lie algebras. Since Lemma~\ref{L2} entails that each subspace $[\LL_\alpha,\LL_{-\alpha}]$ is spanned by a toral element of $\LL$, this is, in fact, an isomorphism of restricted Lie algebras. 
	
	Let $\{\alpha_1,\ldots\alpha_{kp-1}\}\subseteq \Pi$ be the simple roots of $G_2$ with respect to $T$ numbered as in \cite[Planche~I]{Bour}.
	Since 
	$\g_2\cong\mathfrak{pgl}_{kp}$ there exists $h\in \Lie(T)$ such that $[h,\g_1]=0$,  
	$\alpha_1(h)=1$, and $\alpha_i(h)=0$ for $i>0$ Since $h\not\in[\g,\g]$ we have that $\g=\k h\oplus [\g,\g]$. Since $\z(\LL)$ is a torus, it is immediate from Proposition~\ref{P1} 
	that there exists a toral element $\hat{h}\in \LL$ which maps onto $h$ under the canonical 
	homomorphism $\LL\twoheadrightarrow \bar{\LL}=\g$. Since $\LL=\k\hat{h}\oplus [\LL,\LL]$ and $[\LL,\LL]\cong\g_1\oplus\sl_{kp}$, applying Lemma~\ref{L2} once again we deduce that $\LL\cong\g_1\oplus\gl_{kp}\cong\Lie(\widetilde{G})$ as restricted Lie algebras. This completes the proof.
	\end{proof}
\subsection{}\label{3.6}  Proposition~\ref{P2} enables us to identify the restricted Lie algebras $\LL$ and $\widetilde{\g}$ and regard $V$ as an irreducible restricted $\widetilde{\g}$-module. We may also assume that $T$ is a maximal torus of the reductive group $\widetilde{G}$. Let $\n_+$ and $\n_-$ denote the $\k$-spans of all $e_\alpha$ and all $e_{-\alpha}$ with $\alpha\in\Phi_+$, respectively.
By Lemma~\ref{L3}, there is a rational action of $T$ on $V$ compatible with that of $\widetilde{\g}$. Since $V$ is irreducible, the fixed-point space $V^{\n_+}=\{v\in V\,|\,\, \n_+.\, v=0\}$ has dimension $1$ and is spanned by a highest weight vector $v_\lambda$ for $T$, where $\lambda\in X(T)$. Also, $V=U_0(\n_-)\,.\,v_\lambda$. Since the derived subgroup of $\widetilde{G}$ is simply connected, \cite[Theorem~2]{Jan00} shows that $\lambda\in X(T)$ can be chosen to be dominant and $p$-restricted, that is  $0\le \la\lambda,\alpha^\vee\ra\le p-1$ for all $\alpha\in\Pi$.
We denote by $L(\lambda)$ the irreducible rational $\widetilde{G}$-module with highest weight 
$\lambda$ and write $\rho$ for the corresponding representation of $\widetilde{G}$ in 
$\GL(L(\lambda))$. By the general theory of linear algebraic groups, 
the differential ${\rm d}_e\rho\colon\,\widetilde{\g}\to \gl(L(\lambda))$ is a restricted representation of $\widetilde{\g}$.
\begin{Lemma}\label{L6}
	We have that $V\cong L(\lambda)$ as $\widetilde{\g}$-modules. 
\end{Lemma}
\begin{proof}
Since $\lambda\in X(T)$ is a $p$-restricted dominant weight, \cite[Part~II, Prop.~9.24(b)]{Jan03} 
yields that
the $\widetilde{\g}$-module $L(\lambda)$ is irreducible. On the other hand, it is well--known that the irreducible restricted $\widetilde{\g}$-modules are determined up to isomorphism by their highest weights; see \cite[Part~II, Prop.~3.10]{Jan03}, for example. This shows that the restricted $\widetilde{\g}$-modules $V$ and $L(\lambda)$ are isomorphic.
\end{proof}
\subsection{} By our discussion in Subsection~\ref{3.5}, the derived subgroup of $\widetilde{G}$ is simply connected and the Lie algebra $\Lie(\widetilde{G})$ admits a non-degenerate 
$(\Ad \widetilde{G})$-invariant symmetric bilinear form.  In view of Lemma~\ref{L6},  in order to finish the proof of Theorem~\ref{main-thm} it remains show that $p$ is a good prime for $\widetilde{G}$. As $p>3$ by our general assumption, we just need to rule out the case where $p=5$ and $\widetilde{G}$ has components of type ${\rm E}_8$.
\begin{Lemma}\label{L7}
	If $\widetilde{G}$ has a component of type ${\rm E}_8$ then $p>5$.
	\end{Lemma}
\begin{proof}
Let $H$ be a simple component of type ${\rm E}_8$ in $\widetilde{G}$. It is immediate from our description of $\widetilde{G}$ that it contains 
a connected normal subgroup $H'$ such that $\widetilde{G}\cong H\times H'$. Let $\h=\Lie(H)$ and $\h'=\Lie(H')$. The ideals $\h$ and $\h'$ of $\widetilde{\g}$ commute and $\widetilde{\g}=\h\oplus \h'$.  To ease notation we identify $\widetilde{\g}$ with $\LL$. Since $\h$ is a direct summand of  $\widetilde{\g}$, the $\widetilde{\g}$-module $V$ is Richardson for $\h$. 
Let $\widetilde{R}$ be a subspace of $\gl(V)$ such that $[\h,\widetilde{R}]\subseteq \widetilde{R}$ and  $\gl(V)=\h\oplus \widetilde{R}$. 

It is well--known that the nilpotent variety $\N(\h)$ is irreducible and contains a unique open $(\Ad H)$-orbit $\O_{\rm reg}$. Pick $e\in\O_{\rm reg}$ and regard it as a nilpotent  element of $\gl(V)$.
Since the closed subgroup $\rho(H)$ of $\GL(V)$ normalises $\h=({\rm d}_e\rho)(\h)$, the tangent space $T_e(\Ad \rho(H)\,.\,e)$ is contained in $\h$. On the other hand,  since all adjoint  $\GL(V)$-orbits are smooth,
$$T_e(\Ad \rho(H)\,.\,e)\subseteq\,\h\cap T_e((\Ad \GL(V)\,.\,e)=\,\h\cap [\gl(V),e]\,=\,\h\cap\big([\h,e]\oplus [\widetilde{R},e]\big)=\,[\h,e],$$ for $[\h,e]\subseteq \h$ and $\h\cap [\widetilde{R},e]\subseteq \h\cap\widetilde{R}=\{0\}$. Therefore,
$\dim\,(\Ad H)\,.\,e\le \dim\, [\h,e]=\dim\h-\dim \h_e,$ forcing $\dim C_H(e)\ge \dim\h_e$. As $\Lie(C_H(e))\subseteq \h_e$, we deduce the adjoint $H$-orbit of $e\in\O_{\rm reg}$ is smooth. Thanks to a well-known result of Springer, this entails that $p$ is a good prime for $H$; see \cite[Theorem~5.9]{Spr66}. 
\end{proof}
This completes the proof of Lemma~\ref{L7}, and Theorem~\ref{main-thm} follows.
\subsection{} In this subsection we assume that $G$ is a simple, simply connected algebraic $\k$-group and $p$ is a very good prime for $G$.  All standard results on representations of reductive groups used in what follows can be found in \cite{Jan03}.
Given a dominant weight $\lambda\in X(T)$ we denote by
$V(\lambda)$ the Weyl module of highest weight $\lambda$. The quotient 
$L(\lambda):=V(\lambda)/{\rm rad}\,V(\lambda)$ is a simple rational $G$-module of highest weight $\lambda$ and
$\dim V(\lambda)$ can be computed by using the Weyl dimension formula. The Lie algebra $\g=\Lie(G)$ acts on $L(\lambda)$ via the differential at $e\in G$ of the irreducible representation $\rho_\lambda\colon\,G\to GL(L(\lambda))$ and ${\rm d}_e\rho_\lambda\colon\,\g\to\gl(L(\lambda))$  is irreducible if and only if $\lambda$ is $p$-restricted (as defined in Subsection~\ref{3.6}). 
As before, we adopt Bourbaki's numbering of simple roots in $\Pi$ and let $\theta$ be the highest root of $\Phi$ with respect to $\Pi$. Since $\rho$ is already engaged we write $\delta$ for the half-sum of the roots in $\Phi_+$.

Given a finite dimensional representation $\rho\colon\,\g\to  \gl(M)$ we denote by 
${\rm tr}_M(h_\theta)^2$ the trace of the endomorphism $\rho(h_\theta)^2$. 
It is not hard to see for any two finite dimensional restricted $\g$-modules $M$ and $N$ one has
\begin{equation}\label{E1}
{\rm tr}_{M\otimes N}\,(h_\theta)^2\,=\,(\dim\,M)\cdot	{\rm tr}_N(h_\theta)^2+(\dim\,N)\cdot	{\rm tr}_M(h_\theta)^2
\end{equation}
(one should keep in mind that $M$ and $N$ decompose into a direct sum of weight spaces with restpect to $\t$ and $h_\theta\in [\g,\g]$ has zero trace on both $M$ and $N$).

Our goal in this subsection is to find a $p$-restricted dominant $\lambda\in X(T)$ such that $p\nmid\dim L(\lambda)$ and 
${\rm tr}_{L(\lambda)}(h_\theta)^2\ne 0$. In the characteristic zero case similar quantities are often
computed by using the notion of {\it Dynkin index}; see \cite[Ch.~I, \S 2]{Dyn}. Therefore, it will be convenient for us to regard $G$ as the group of $\k$-points of  a simply connected Chevalley group scheme $G_\Z$ with the same root datum as $G$. 
We shall also assume that $T$ is obtained by base-changing a maximal split torus of $T_\Z$ of $G_\Z$, so that $X(T)=X(T_\Z)$. Then $h_\theta=H_\theta\otimes_\Z 1$ for the semisimple root vector $H_\theta={\rm d}_e(\theta^\vee)\in\Lie(T_\Z)$. 

Let $\g_\Z$ be the Lie algebra of $G_\Z$ and denote by $V_\Z(\lambda)$ the Weyl 
module for $G_\Z$ with highest weight $\lambda\in X(T)$. Then $\g=\g_\Z\otimes_\Z \k$ and $V(\lambda)=V_\Z\otimes_\Z\k$. We denote by $X(\lambda)$ the set of all $T_\Z$-weights of $V_\Z(\lambda)$ and write $n_\mu$ for the multiplicity of $\mu\in X(\lambda)$ in $V_\Z(\lambda)$. The Dynkin index $d(\lambda)$ of $V_\Z(\lambda)$ is defined as
$$d(\lambda)\,:=\,\frac{1}{2}\sum_{\mu\in X(\lambda)}n_\mu\, \mu(H_\theta)^2\,=\,\frac{1}{2}\,{\rm tr}_{V_\Z(\lambda)}\,(H_\theta)^2.$$
It is well-known that $d(\lambda)$ is an integer; see \cite[2.3]{LaSo}, for example. For all fundamental dominant weights $\varpi_i$ with $1\le i\le {\rm rk}(G_\Z)$, the integers
$d(\varpi_i)$ are computed in \cite[Table~5]{Dyn}, and the three 
misprints in type ${\rm E}_8$ are corrected in \cite[p.~504]{LaSo}. 

Let $(\,\cdot\,,\,\cdot\,)$ be the scalar product on the $X(T_\Z)\otimes_\Z \mathbb{R}$ invariant under the action of the Weyl group $W(\Phi)$ and such that $(\theta,\theta)=2$. Dynkin proved in
\cite{Dyn} that
\begin{equation}\label{E2}
	d(\lambda)\,=\,\frac{\dim\,V_\mathbb{Q}(\lambda)}{\dim\,\g_\mathbb{Q}}\cdot (\lambda+2\delta,\lambda).
\end{equation}	 
Here $V_\mathbb{Q}(\lambda)=V_\Z(\lambda)\otimes_\Z\mathbb{Q}$ and $\g_\mathbb{Q}=\g_\Z\otimes_\Z\mathbb{Q}$. 
Dynkin's original proof of the integrality of $d(\lambda)$ involved some case-by-case considerations, but a shorter argument was later found in \cite[Ch.~I, \S 3.10]{Oni}. We refer to \cite[\S 1]{Pan} for more detail on the history of this formula.
\begin{Proposition}\label{P3}
If $p$ is a very good prime for $G$, then there exists a $p$-restricted dominant weight $\lambda\in X(T)$ such that $p\nmid\dim\,L(\lambda)$ and 
${\rm tr}_{L(\lambda)}\,(h_\theta)^2\ne 0$.
\end{Proposition}		
	\begin{proof}
		If $G$ is $\SL_n$ or $\Sp_{2n}$ with $p\nmid n$, then the natural $G$-module $L(\varpi_1)$ satisfies all our requirements since $p\nmid \dim L(\varpi_1)$ and
		 ${\rm tr}_{L(\varpi_1)}\,(h_\theta)^2=2$. If $G=\Sp_{2n}$ and $p\mid n$ we can take  
		 $L(\varpi_2)$ instead. Indeed, \cite[Cor.~2]{PSup} shows that in this case   
		 ${\rm rad}\,V(\varpi_2)\cong L(0)$, implying that 
		 $\dim\,L(\varpi_2)={2n\choose 2}-2$. Since $d(\varpi_2)=2n-2$ by \cite[p.~504]{LaSo}, we have that $${\rm tr}_{L(\varpi_2)}\,(h_\theta)^2=\,{\rm tr}_{V(\varpi_2)}\,(h_\theta)^2=\,2d(\varpi_2)\bmod p\in{\mathbb F}_p^\times.$$
		 If $G$ is of type ${\rm B}_n$ or ${\rm D}_n$, then \cite[p.~504]{LaSo} shows that we can take for $\lambda$ the minuscule weight $\varpi_n$ as both $d(\varpi_n)$ and 
		 $\dim L(\varpi_n)=\dim V(\varpi_n)$ are powers of $2$ (and $p>2$).
		 
		 Now suppose $G$ is an exceptional group. Since $p$ is a good prime for $G$, the Killing form of $\g$ is non-degenerate. This is well--known and follows, for example, from 
		 the fact that $d(\theta)=2h^\vee$, where $h^\vee$ is the dual Coxeter number of $\Phi$; see \cite[\S 1]{Gar} for more detail.
		 If $p\nmid \dim\,\g$, we can take for $L(\lambda)$ the adjoint $G$-module $\g\cong L(\theta)$. Indeed, our assumptions on $p$ and $G$ imply that 
		 $\g$ is a simple Lie algebra and $p\nmid h^\vee$.
		 
		 From now on we may assume that $G$ is exceptional and $p\mid\dim\,\g$. If $G$ is of type ${\rm E}_6$
		 we can take for $\lambda$ the minuscule weight $\varpi_1$. In this case $V(\varpi_1)\cong L(\varpi_1)$ has dimension $27$ and $d(\varpi_1)=6$ by 
		 \cite[p.~504]{LaSo}. 
		 If $G$ is of type ${\rm E}_7$ and $p=7$ we can take for $\lambda$ the fundamental weight $\varpi_6$. Then $d(\varpi_6)=648$ by \cite[p.~504]{LaSo}, whilst a quick look at \cite[6.52]{Lue} reveals that ${\rm rad}\,V(\varpi_6)\cong L(0)$ has dimension $1$ and $\dim L(\varpi_6)=1538$. Since both $648$ and $1538$ are invertible 
		 in ${\mathbb F}_7$ and 
		 ${\rm tr}_{L(\varpi_6)}(h_\theta)^2={\rm tr}_{V(\varpi_6)}(h_\theta)^2=2d(\varpi_6)\bmod 7,$ this choice of $\lambda$ satisfies all our requirements. If $G$ is of type ${\rm E}_7$ and $p=19$
		 we take for $\lambda$ the minuscule weight $\varpi_7$. Then $V(\varpi_7)\cong L(\varpi_7)$ has dimension $56$ and $d(\varpi_7)=12$ by 
		 \cite[p.~504]{LaSo}. Since both $56$ and $12$ are invertible in ${\mathbb F}_{19}$ this is a good choice for us.
		 
		 If $G$ is of type ${\rm G}_2$, ${\rm F}_4$ or ${\rm E}_8$, then the good primes dividing $\dim \g$ are $7$, $13$ and $31$, respectively. The tables in \cite{Lue} indicate that we cannot use the fundamental highest weights to construct a suitable $L(\lambda)$.
		 
		 Suppose $G$ is of type ${\rm E}_8$ and $p=31$. From \cite[6.53]{Lue} we get $\dim V_\mathbb{Q}(2\varpi_8)=27000$ and $\dim L(2\varpi_8)=26999$. Hence ${\rm rad}\,V(2\varpi_8)\cong L(0)$, implying that 
		 ${\rm tr}_{L(2\varpi_8)}\,(h_\theta)^2=2d(2\varpi_8)\bmod 31$. 
		 Since all roots in $\Phi$ have the same length we can take for $(\,\cdot\,,\,\cdot\,)$ the scalar product from \cite[Planche~VII]{Bour}. Then
		 $$(2\varpi_8+2\delta, 2\varpi_8)\,=\,4(\varepsilon_7+\varepsilon_8+\delta\mid\varepsilon_7+\varepsilon_8)\,=\,
		 4(\varepsilon_2+2\varepsilon_3+3\varepsilon_4+4\varepsilon_5+5\varepsilon_6+7\varepsilon_7+24\varepsilon_8
		 \mid\varepsilon_7+\varepsilon_8)\,=\,124.$$
		 Using (\ref{E2}) we obtain $2d(2\varpi_8)=\frac{27000}{248}\cdot 248=27000.$ Since both $26999$ and $27000$ are invertible in $\mathbb{F}_{31}$, the module $L(2\varpi_8)$ satisfies all our requirements.
		 
		 Suppose $G$ is of type ${\rm F}_4$ and $p=13$. By \cite[Planche~VIII]{Bour}, we can again choose for $(\,\cdot\,,\,\cdot\,)$ the scalar product $(\,\cdot\mid\cdot\,)$. Setting $\lambda:=2\varpi_4$ we get	$$(\lambda+2\delta,\lambda)\,=\,
		 (2\varepsilon_1+11\varepsilon_1+5\varepsilon_2+3\varepsilon_3+\varepsilon_4\mid 2\varepsilon_1)=26.$$ By \cite[6.50]{Lue}, we have that 
		 $\dim V_\mathbb{Q}(\lambda)=324$ and $\dim L(\lambda)=323$, which implies that ${\rm rad}\,V(\lambda)\cong L(0)$ and ${\rm tr}_{L(\lambda)}\,(h_\theta)^2=2d(\lambda)\bmod 13$. 
		 Since $2d(\lambda)=\frac{324}{52}\cdot 52=324$ by (\ref{E2}) and both $323$ and $324$ are invertible in $\mathbb{F}_{13}$, this choice of $\lambda$ is good for us.
		 
		 Suppose $G$ is of type ${\rm G}_2$ and $p=7$. By \cite[Planche~IX]{Bour}, we can choose for $(\,\cdot\,,\,\cdot\,)$ the scalar product $\frac{1}{3}(\,\cdot\mid\cdot\,)$.
		 Set $\lambda:=2\varpi_1$. Keeping in mind the numbering of simple roots in  \cite[6.49]{Lue} we get $\dim V(\lambda)=27$ and $\dim L(\lambda)=26$.  So ${\rm rad}\,V(\lambda)\cong L(0)$, forcing ${\rm tr}_{L(\lambda)}\,(h_\theta)^2=2d(\lambda)\bmod 7$. Note that $\lambda=2\theta_0$, where  $\theta_0$ is the highest short root of $\Phi_+$, and $\delta=\theta_0+\theta$. Then 
		 $$(\lambda+2\delta,\lambda)\,=\,\frac{4}{3}(\theta_0+\delta\mid\theta_0)\,=\,
		 \frac{4}{3}(2\theta_0+\theta\mid\theta_0)\,=\,\frac{4}{3}(4+3)\,=\,\frac{28}{3}.$$
		 Then (\ref{E2}) gives  $2d(\lambda)=2\cdot\frac{27}{14}\cdot \frac{28}{3}=36$. As both $26$ and $36$ are invertible in $\mathbb{F}_{7}$, this choice of $\lambda$ is good for us. The proof of the proposition is now complete.	
	\end{proof}
	\begin{Corollary}\label{C2}
	Suppose $G$ is a standard reductive $\k$-group such that $p$ is a very good prime for $G$ and $\dim Z(G)\le 1$. Then there exists an infinitesimally irreducible rational representation $\rho\colon\, G\to\GL(V)$ with $p\nmid \dim V$ such that ${\rm d}_e\rho$ is a 
	faithful irreducible representation of $\g=\Lie(G)$ and the trace form 
	$(X,Y)\mapsto {\rm tr}\big(({\rm d}_e\rho)(X)\circ({\rm d}_e\rho)(Y)\big)$ on $\g$ is non-degenerate.
	\end{Corollary}
\begin{proof}
	Since the derived subgroup $\D G$ of $G$ is simply connected we have that $G=Z(G)\cdot\D G$ and  
	$\D G\cong  G_1\times\cdots\times G_s$ where $G_1, \ldots, G_s$ are the simple components of $G$.
	For each $i\le s$ we choose a $p$-restricted dominant weight $\lambda_i$ of $G_i$ 
	satisfying the requirements of Proposition~\ref{P3} and consider the irreducible $\D G$-module
	$V:=\,L(\lambda_1)\otimes\cdots\otimes L(\lambda_s)$. We allow $Z(G)$ to act on $V$ by scalar operators in such a way that the differential of the action is nonzero, and denote by $\rho$ the infinitesimally irreducible representation of $G$ obtained this way. The Lie algebra 
	$\g=\Lie(Z(G))\oplus \sum_{i=1}^s\,\Lie(G_i)$ then acts faithfully  
	on $V$. Since $p\nmid \dim L(\lambda_i)$ for all $i$, it is immediate from (\ref{E1}) that
	the restriction of the symmetric bilinear form $(X,Y)\mapsto {\rm tr}\big(({\rm d}_e\rho)(X)\circ ({\rm d}_e\rho)(Y)\big)$ on $\g$ to each $\Lie(G_i)$ and to $\Lie(Z(G))$is nonzero. Since each ideal $\Lie(G_i)$ is a simple Lie algebra, the result follows. 
	\end{proof}
\begin{Remark}
Let $R$ to be the orthogonal complement of $({\rm d}_e\rho)(\g)$ with respect to the non-degenerate trace form associated with $V$. Then $\gl(V)=({\rm d}_e\rho)(\g)\oplus R$ and $[({\rm d}_e\rho)(\g)\, R]\subseteq R$. Since ${\rm d}_e\rho$ is faithful, we see that $V$ is a Richardson module for $\g$. Furthermore, the complement $R$ is invariant under the conjugation action of $\rho(G)$ on $\gl(V)$. 
\end{Remark}
\begin{Remark}
Suppose $G$ is as in Corollary~\ref{C2} (with $s\ge 1$) and consider 
$\widetilde{G}=\D G\times GL_{kp}$ where $k\ge 1$.
It is well--known (and straightforward to see) that there exist $u,v\in\N_p(\gl_{kp})$ such that
$[u,v]=1_{kp}$. Let $\mathfrak{a}$ denote the $3$-dimensional restricted Lie subalgebra of $\gl_{kp}$ spanned by $u$, $v$ and $1_{kp}$. Representation theory of nilpotent Lie algebras
shows that all faithful irreducible $\mathfrak{a}$-modules have the same dimension equal to $p$.
Consequently, any faithful irreducible  $\gl_{kp}$-module has dimension divisible by $p$.  
Using (\ref{E1}) one observes that
 in contrast with Corollary~\ref{C2} the trace forms associated with the irreducible faithful representations of $\Lie(\widetilde{G})$ are always degenerate. 
 \end{Remark}
\begin{Remark} 
	Suppose $\g=\sl_2$ and $p=3$. As mentioned at the end of Subsection~\ref{S0} any Weyl module $V(m)$ with $m\in\Z_{\ge 0}$  is a Richardson module for $\g$. It follows from (\ref{E2}) that $d(m)=\frac{m+1}{3}\cdot\frac{1}{2} m(m+2)={m+2\choose 3}$. Therefore, if $m(m+1)(m+2)$ is divisible by $9$, then the trace form of $\gl(V(m))$ vanishes on $({\rm d}_e\rho_m)(\g)$. This example shows that there exist indecomposable Richardson modules $V$ for a restricted Lie algebra $\LL$ such that the trace form of $\gl(V)$ vanishes on $\LL\subset\gl(V)$.
	\end{Remark}

\bibliographystyle{amsalpha}

\end{document}